\documentclass[11pt,a4paper]{article}
\usepackage[utf8]{inputenc}
\usepackage[T1]{fontenc}

\usepackage{fourier} 

\usepackage{amsmath}
\usepackage{amsfonts}
\usepackage{amssymb}
\usepackage{amsthm}
\usepackage{ upgreek }

\usepackage{marginnote}

\usepackage{csquotes}

\usepackage{stmaryrd}

\usepackage{enumitem}

\usepackage{mathtools}

\usepackage{cleveref}
\usepackage{autonum} 
\usepackage{calc}

\usepackage{pgf,tikz}
\usetikzlibrary{matrix,automata,arrows,positioning,calc,patterns,fit,calc,
decorations.pathreplacing,decorations.pathmorphing,decorations.markings,matrix,
tikzmark,shadows.blur}
\usetikzlibrary{cd}
 \usepgflibrary{fpu}

\usepackage{filecontents}

\crefformat{equation}{(#2#1#3)}

\newtheorem{theorem}{Theorem}[section]
\newtheorem{proposition}[theorem]{Proposition}
\newtheorem{lemma}[theorem]{Lemma}

\newtheorem{definition}[theorem]{Definition}

\theoremstyle{definition}

\newtheorem{remarkx}[theorem]{Remark}
\newenvironment{remark}
  {\pushQED{\qed}\remarkx}
  {\popQED\endremarkx}

\numberwithin{equation}{section}

\DeclareMathOperator{\dd}{d}            
\DeclareMathOperator{\e}{e}             
\DeclareMathOperator{\Var}{Var}         

\newcommand{\jbrak}[1]{\langle#1\rangle}
\newcommand{\eps}{\varepsilon} 

\newcommand{\6}[1]{\dd\!#1}             

\newcommand{\field}[1]{\mathbb{#1}}

\newcommand{\Z}{\field{Z}}    
\newcommand{\R}{\field{R}}    
\newcommand{\E}{\field{E}}      
\newcommand{\fP}{\field{P}}     
\newcommand{\T}{\field{T}}      

\newcommand{\cB}{\mathcal{B}}   
\newcommand{\cC}{\mathcal{C}}   
\newcommand{\cO}{\mathcal{O}}   

\begin{document}


\title{Concentration estimates for SPDEs\\ driven by fractional Brownian motion}
\author{Nils Berglund, Alexandra Blessing (Neam\c tu)}
\date{}   

\maketitle

\begin{abstract}
The main goal of this work is to provide sample-path estimates for 
the solution of slowly time-dependent SPDEs perturbed by a cylindrical fractional 
Brownian motion. Our strategy is similar to the approach by Berglund and Nader 
for space-time white noise. However, the setting of fractional Brownian motion 
does not allow us to use any martingale methods. Using instead optimal estimates 
for the probability that the supremum of a Gaussian process exceeds a certain level, 
we derive concentration estimates for the solution of the SPDE, provided that the Hurst 
index $H$ of the fractional Brownian motion satisfies $H>\frac14$. As a by-product, 
we also obtain concentration estimates for one-dimensional fractional SDEs valid 
for any $H\in(0,1)$.
\end{abstract}

\leftline{\small{\it Date.\/} April 25, 2024. Revised version from December 29, 2024.}
\noindent{\small 2020 {\it Mathematical Subject Classification.\/} 
60G15, 60G17, 60H15.
}

\noindent{\small{\it Keywords and phrases.\/}
Concentration estimates, slow-fast systems, fractional Brownian motion, SPDEs.
}  

\section{Introduction}

Fractional Brownian motion (fBm) is a famous example of stochastic process used in 
order to model memory effects or long-range dependencies. An fBm is a centered, stationary 
Gaussian process parameterized by a so-called Hurst index/parameter $H\in(0,1)$. For 
$H = \frac12$, one recovers the classical Brownian motion. However, for $H \in (\frac12, 1)$ 
and $H \in (0, \frac12)$, fBm exhibits a different behaviour than Brownian motion. 
Its increments are no longer independent, but positively correlated for $H > \frac12$, 
and negatively correlated for $H < \frac12$. Fractional Brownian motion has been used to model 
a wide range of phenomena, extending from mathematical finance~\cite{Rogers} to fluid 
dynamics~\cite{Weiss13}.

Since fractional Brownian motion is neither a Markov process, nor a semi-martingale, it is a 
challenging task to construct solutions of SDEs/SPDEs driven by such a process and to analyze 
their dynamical properties~\cite{HaLi,HaOh}. Here we contribute to this topic by deriving  
concentration estimates for slowly time-dependent SDEs perturbed by an additive fBm for all 
ranges of the Hurst index $H\in(0,1)$, and for slowly time-dependent semilinear SPDEs perturbed 
by a cylindrical fBm, provided $H\in(\frac14,1)$. 
In this case, the well-posedness of the SPDEs is well-known~\cite{DuncanMaslowski, DuncanMaslowski2,DMD3,TTV}. We mention that concentration estimates for SDEs driven by fBm with $H\in(\frac12,1)$ were 
previously obtained in~\cite{EKN}. One of the main novelties of this work is to extend this 
finite-dimensional result to $H\in(0,\frac12)$. The tools used in~\cite{EKN}, based on 
results in~\cite{D}, break down in this case, since they require the covariance of the 
fBm to be increasing, and rely on Lyapunov-type equations for the variance of a non-autonomous 
fractional Ornstein--Uhlenbeck process.~The SDEs we consider in this work have the form
\begin{equation}
\label{eq:SDE_intro} 
\6x_t = f(\eps t,x_t) \6t + \sigma \6W^H_t\;, 
\end{equation} 
where the drift term $f$ admits a so-called stable uniformly hyperbolic critical manifold, that is, a smooth curve on which $f$ vanishes and has a uniformly negative $x$-derivative. In the deterministic case $\sigma = 0$, it is well known \cite{Tihonov,Fenichel} that the equation admits a so-called slow solution, staying $\eps$-close to such a curve, and attracting nearby solutions exponentially fast. In the Brownian case $H=\frac12$, it was shown in~\cite{BG02} that for $\sigma>0$, sample paths are concentrated in a neighbourhood of size of order $\sigma$ of such a slow solution. Our first main result, Theorem~\ref{c:sde:nl}, provides similar concentration estimates for any $H\in(0,1)$, with explicit bounds on the probability of leaving such a neighbourhood. We believe that these results can be extended to more general Gaussian processes satisfying a self-similarity property. The SPDEs we consider have the form 
\begin{equation}
\label{eq:spde_intro}
 \6\phi(t,x) =  \bigl[\Delta \phi(t,x) + f(\eps t,\phi(t,x))\bigr]\6t 
 + \sigma \6W^H(t,x)\;,
\end{equation}
where $x$ belongs to the one-dimensional torus $\T$, and $W^H(t,x)$ denotes a cylindrical fractional 
Brownian motion on $\T$. In the Brownian case $H=\frac12$, concentration estimates near stable uniformly hyperbolic critical manifolds have been obtained in~\cite{BN23} for the one-dimen\-sional torus, 
and in~\cite{BN:22x} for the two-dimen\-sional torus $\T^2$, provided the equation is suitably renormalised. 
Our second main result, Theorem~\ref{thm:nonlinear_concentration}, extends the concentration results 
on $\T$ to all $H\in(\frac14,1)$, for all fractional Sobolev norms of index $s\in(0,2H-\frac12)$. 
 
Numerous extensions and applications of these results are imaginable. For instance, one could investigate bifurcations in SDEs/SPDEs with fractional noise. For example~\cite{BlBl23} analyzes pitchfork bifurcations using finite-time Lyapunov-exponents and approximations with amplitude equations derived in~\cite{NBl}, whereas~\cite{LKN} computes early-warning signs for fast-slow systems perturbed by additive fractional noise. These results have been used in~\cite{LKN} to derive scaling laws for bifurcations for the Stommel--Cessi model for the Atlantic Meridional Overturning Circulation~\cite{Stommel,Cessi}. In the Brownian case $H=\frac12$, the works~\cite{BN23,BN:22x} have obtained similar results for pitchfork and avoided transcritical bifurcations in slowly time-dependent SPDEs. Another exciting direction is given by concentration estimates for SDEs/SPDEs perturbed by multiplicative fractional noise, using tools from rough path theory and slow-fast systems~\cite{HaLi}.

This manuscript is structured as follows. Section~\ref{sec:prelim} gives a precise definition of fBm, and provides a criterion allowing to estimate the probability that the supremum of a mean-square H\"older continuous Gaussian process exceeds a certain level. Section~\ref{sec:SDE} deals with sample-path estimates for one-dimensional SDEs. The key result in this setting is a suitable upper bound on the variance of a non-autonomous fractional Ornstein--Uhlenbeck process. The concentration inequality is extended to semilinear SPDEs in 
Section~\ref{sec:SPDE}, by analyzing the Fourier components of its solution and using Schauder-type estimates. 

\section{Fractional Brownian motion}
\label{sec:prelim} 

In this section we collect basic results on fractional Brownian motion and Gaussian 
processes which will be required later on. 

\begin{definition}
A fractional Brownian motion (fBm) $(W^{H}_t)_{t\geq 0}$ with Hurst index $H\in(0,1)$ is a centered Gaussian process with covariance 
\begin{equation}
\label{eq:cov_fBm} 
\E [W^H_tW^H_s]=\frac{1}{2} (t^{2H} +s^{2H} -|t-s|^{2H})\;.
 \end{equation}
\end{definition}

For $H=\frac12$ we recover the standard Brownian motion, whereas for $H\neq \frac12$ we obtain a process which is neither Markov nor a semi-martingale. 

A useful estimate for the probability that the supremum of a Gaussian process exceeds a certain threshold is given by the following theorem~\cite[Theorem D.4]{P}. It is obtained by comparing the probability of exceeding a certain level with the one of a suitable stationary process, using Slepian's lemma~\cite{Slepian}. The mean-square H\"older continuity  enables one to define such a process. 

\begin{theorem}
\label{gaussineqforHolderCont}
Let $(X_t)_t$ be a continuous Gaussian process with zero mean on $[0,T]$ for $T>0$.
Assume that $(X_t)_t$ is mean-square H\"older continuous, i.e., there are constants $G$ and $\gamma$ such that
\begin{equation}
\E \left[ (X_t - X_s)^2 \right] \leq G \left| t - s \right|^{\gamma} \hspace{0.2cm} 
\qquad \text{for all }t,s \in [0,T]\;.
\end{equation}
Then there exists a constant $K := K(G,\gamma)$ such that for $c > 0$ and $A \subset [0,T]$, 
one has 
\begin{equation}
\mathbb{P} \left\{ \sup_{t \in A} X_t > c\right\} 
\leq K T c^{2/\gamma}\exp\left\{- \frac{c^2}{2\sigma^2(A)} \right\}\;,
\end{equation}
where $\sigma^2(A) := \sup_{t \in A} \text{\em Var} \left( X_t\right)$.
\end{theorem}

\begin{remark}
\label{scale:GK}
By a simple scaling argument, one can infer that 
\begin{equation}
 K(G,\gamma) = G^{-1/\gamma} K(1,\gamma) =: G^{-1/\gamma} K_0(\gamma)\;.
\end{equation} 
Indeed, the process $\tilde X_t = G^{-1/2} X_t$ has H\"older constant $1$,  
maximal variance $G^{-1} \sigma^2(A)$, and satisfies 
$\sup_{t\in A}\tilde X_t = G^{-1/2}\sup_{t\in A}\tilde X_t$.
\end{remark}

In our case, we use this result for a non-autonomous fractional Ornstein-Uhlenbeck process 
of Hurst index $H\in(0,1)$, which is known to be mean-square H\"older continuous with 
exponent $\gamma = 2H$. 

\section{The one-dimensional SDE case}
\label{sec:SDE} 

\subsection{Linear case}

We start by considering linear fractional SDEs, driven by an fBm $(W^H_t)_{t\geq 0}$ with 
Hurst parameter $H\in(0,1)$, given by 
\begin{equation}
\label{eq:SDE_linear_fast} 
\6x_t = a(\eps t) x_t \6t + \sigma \6W^H_t\;, 
\end{equation} 
where $\eps, \sigma>0$ are small parameters, and $a:[0,T]\to\R$ is of class $\cC^1$.
It will be convenient to scale time by a factor $\eps$, which turns the SDE~\eqref{eq:SDE_linear_fast}
into 
\begin{equation}
\label{eq:SDE_linear} 
\6x_t = \frac{1}{\eps} a(t) x_t \6t + \frac{\sigma}{\eps^{H}} \6W^H_t\;.
\end{equation} 
We will assume that the function $a$ satisfies 
\begin{equation}
\label{eq:assump_a} 
 a(t) \leq -a_0\;, \qquad 
 |a'(t)| \leq a_1 \qquad 
 \forall t\in[0,T]
\end{equation} 
for some constants $a_0, a_1 > 0$, and write 
\begin{equation}
\alpha(t) = \int_0^t a(s)\6s\;, \qquad 
\alpha(t,u) = \int_u^t a(s)\6s\;.
\end{equation} 
In order to apply Theorem~\ref{gaussineqforHolderCont}, we will need to control the 
variance of $x_t$. The following result is an adaptation of~\cite[Theorem~1.43]{Mishura} to the non-autonomous case.

\begin{lemma}\label{variance}
Assume the initial condition $x_0$ in~\eqref{eq:SDE_linear} is deterministic. 
For any $H\in(0,1)$, the variance of $x_t$ satisfies the upper bound
\begin{equation}\label{eq:variance1}
 \Var(x_t) \leq \frac{2H\sigma^2}{\eps^{2H}}
 \int_0^t 
 \Bigl[ \e^{\alpha(t,s)/\eps}   (t-s)^{2H-1} 
 - \e^{\alpha(t)/\eps} (1 - \e^{\alpha(t,s)/\eps}) s^{2H-1} \Bigr]  \6s\;.
\end{equation} 
\end{lemma}
\begin{proof}
The proof is based on the representation 
\begin{equation}
 x_t = x_0\e^{\alpha(t)/\eps} + \frac{\sigma}{\eps^H} 
 \int_0^t \frac{a(s)}{\eps} \e^{\alpha(t,s)/\eps} W^H(s)\6s 
 + \frac{\sigma}{\eps^H} W^H(t)
\end{equation} 
obtained by integration by parts. Since $x_0\e^{\alpha(t)/\eps}$ is deterministic, we obtain 
\begin{align}
\label{variance:non}
\Var(x_t) = \frac{\sigma^2}{\varepsilon^{2H}}\Bigl[& \E[(W^H(t))^2] 
+ 2 \int_0^t \frac{a(s)}{\varepsilon}\e^{\alpha(t,s)/\varepsilon} \E[W^H(t)W^H(s)]\6s \\
&{}+  
\int_0^t \int_0^t \frac{a(u)}{\varepsilon} \frac{a(v)}{\varepsilon} 
\e^{\alpha(t,u)/\varepsilon}\e^{\alpha(t,v)/\varepsilon}\E[W^H(u)W^H(v)]~\6u~\6v \Bigr]\;. 
\end{align}
By the expression~\eqref{eq:cov_fBm} of the covariance function of the fBm, the first term in square brackets gives $\E[(W^H(t))^2]=t^{2H}$, while the second term becomes
\begin{equation}
2\int_0^t \frac{a(s)}{\varepsilon}\e^{\alpha(t,s)/\varepsilon} \E[W^H(t)W^H(s)]\6s
= \int_0^t \frac{a(s)}{\varepsilon}\e^{\alpha(t,s)/\varepsilon} [t^{2H} +s^{2H}-(t-s)^{2H}]\6s\;.
\end{equation}
We split this into three integrals that we compute separately. The first one gives  
\begin{equation}
I_1 := t^{2H} \int_0^t \frac{a(s)}{\varepsilon} \e^{\alpha(t,s)/\varepsilon}~\6 s
= t^{2H}[\e^{\alpha(t)/\varepsilon}-1]\;, 
\end{equation}
while the other two integrals can be evaluated using integration by parts, yielding 
\begin{align}
I_2 &:= \int_0^t \frac{a(s)}{\varepsilon} \e^{\alpha(t,s)/\varepsilon} s^{2H}~\6s 
=-t^{2H} + 2H \int_0^t \e^{\alpha(t,s)/\varepsilon}s^{2H-1}~\6s\;,\\
I_3 &:= -\int_0^t \frac{a(s)}{\varepsilon}\e^{\alpha(t,s)/\varepsilon}(t-s)^{2H}~\6s 
=-\e^{\alpha(t)/\varepsilon} t^{2H} + 2H \int_0^t \e^{\alpha(t,s)/\varepsilon} (t-s)^{2H-1}~\6s\;.
\end{align}
We further split the last term in the expression~\eqref{variance:non} for the variance of $x_t$ into three integrals. By symmetry, the first two are equal and add up to  
\begin{align}
I_4 : ={}& 
\int_0^t\int_0^t \frac{a(u)}{\varepsilon}\frac{a(v)}{\varepsilon}\e^{\alpha(t,u)/\varepsilon}\e^{\alpha(t,v)/\varepsilon}u^{2H}~\6u~\6v \\
={} & t^{2H} -t^{2H}\e^{\alpha(t)/\varepsilon} + 2H (\e^{\alpha(t)/\varepsilon}-1) 
\int_0^t \e^{\alpha(t,s)/\varepsilon}s^{2H-1}~\6s\;.
\end{align}
The last integral, given by
\begin{equation}
-\int_0^t\int_0^t \frac{a(u)}{\varepsilon} \frac{a(v)}{\varepsilon} \e^{\alpha(t,u)/\varepsilon}
\e^{\alpha(t,v)/\varepsilon}|u-v|^{2H}~\6u~\6v\;,
\end{equation}
can be dropped since $a$ is negative. The result follows by exploiting cancellations, 
and writing the factor $t^{2H}$ occurring in one remaining term $-t^{2H}\e^{\alpha(t)/\eps}$ as the integral of $2Hs^{2H-1}$.
\end{proof}

Laplace asymptotics and our assumptions~\eqref{eq:assump_a} on $a$ allow us to obtain the following simplified expression for the variance for small $\eps$.  

\begin{lemma}
There exists a constant $r_1$, depending only on $a_0$, $a_1$ and $H$, such that for any 
$H\in(0,1)$ and any $t\geq0$, one has 
\begin{equation}
\label{estimate:variance}
\Var(x_t)\leq \frac{\sigma^2 2H\Gamma(2H)}{|a(t)|^{2H}} 
(1+r_1\varepsilon)\;. 
\end{equation}
\end{lemma}
\begin{proof}
It is sufficient to bound the first term in the integrand in~\eqref{eq:variance1}, 
since the second term is negative. The assumptions~\eqref{eq:assump_a} on $a$ imply that 
whenever $0\leq s\leq t\leq T$, one has 
\begin{equation}
\alpha(t,s)\leq -a_0 (t-s)
\qquad\text{and}\qquad 
\alpha(t,s) \leq a(t)(t-s) + \frac12 a_1 (t-s)^2\;.
\end{equation}
Using the substitution $y=a(t)(t-s)$, we obtain 
\begin{equation}
\int_0^t \e^{\alpha(t,s)/\varepsilon}(t-s)^{2H-1}~\6s 
=\frac{1}{|a(t)|^{2H}} \int_0^{|a(t)| t} \e^{\alpha(t,t-\frac{y}{|a(t)|})/\varepsilon} y^{2H-1}~\6y\;.
\end{equation}
If $|a(t)| t > \frac{a_0}{a_1}$, we split the integral at $y = \frac{a_0}{a_1}$. To bound the integral 
over the interval $[0, \frac{a_0}{a_1}]$, we use the fact that 
\begin{equation}
 \alpha\Bigl(t,t-\frac{y}{|a(t)|}\Bigr)  
 \leq - y + \frac{a_1}{2a(t)^2}y^2
 \leq -y + \frac{a_1}{2a_0}y^2\;,
\end{equation}
to obtain 
\begin{align}
\int_0^{a_0/a_1} \e^{\alpha(t,s)/\varepsilon}(t-s)^{2H-1}~\6s  
&\leq \frac{1}{|a(t)|^{2H}} \int_0^{a_0/a_1} y^{2H-1} 
\exp\biggl\{-\frac{1}{\eps}\Bigl[y - \frac{a_1}{2a_0}y^2\Bigr]
\biggr\}~\6y 
\\
&=\frac{\varepsilon^{2H}}{|a(t)|^{2H}} \int_0^{a_0/\eps a_1} z^{2H-1} 
\exp\biggl\{-\Bigl[z - \eps\frac{a_1}{2a_0}z^2\Bigr]
\biggr\}~\6z\;,
\end{align}
where we have set $y=\eps z$. 
Without the term $-\eps\frac{a_1}{2a_0}z^2$, the integral would be bounded above 
by $\Gamma(2H)$. Using results on Laplace asymptotics, see for instance~\cite[Theorem~8.1]{Olver}, 
one obtains the upper bound $\Gamma(2H)[1 + \mathcal{O}(\eps)]$.

Using the fact that $\alpha(t,t-\frac{y}{|a(t)|}) \leq -y$, one finds that 
the integral over the remaining interval $[\frac{a_0}{a_1},|a(t)| t]$ 
is exponentially small in $\eps$, and therefore negligible with respect to the error of 
order $\eps$. Finally, if $|a(t)| t < \frac{a_0}{a_1}$, we can use the integral over 
$[0, \frac{a_0}{a_1}]$ as an upper bound.
\end{proof}

\begin{remark}
For $H=\frac12$, this result is consistent with the case of Brownian motion investigated by Berglund and Gentz in~\cite{BG02}. In particular, the $H$-dependent constant in~\eqref{estimate:variance} is given by $2H\Gamma(2H)=1$ for $H=\frac12$. 
\end{remark}

As a first consequence of the bound~\eqref{estimate:variance}, we obtain a concentration 
result for the solutions of the linear SDE~\eqref{eq:SDE_linear}. 
This is based on Theorem~\ref{gaussineqforHolderCont}, taking into account 
the scaling argument in Remark~\ref{scale:GK}.  

\begin{proposition}[Concentration estimate for the linear SDE]
\label{prop:lin:stable}
Assume $x_0 = 0$. Then there exists a constant $r_2$, depending only on $a_0$, $a_1$ and $H$, such that 
\begin{equation}
\label{eq:prop_linear_stable} 
 \fP\biggl\{ \sup_{0 \leq t \leq T} |x_t| |a(t)|^H \geq h \biggr\}
 \leq C\Bigl(T; \frac{h}{\sigma}, a_0\Bigr) 
 \exp\biggl\{ -\kappa(\eps)\frac{h^2}{2\sigma^2}\biggr\}\;,
\end{equation} 
where the prefactor and exponent are given by 
\begin{equation}
\label{eq:factors_linear_stable} 
 C\Bigl(T; \frac{h}{\sigma}, a_0\Bigr) = 
 \frac{2 K_0(2H)T^2}{a_0} \Bigl( \frac{h}{\sigma} \Bigr)^{1/H}\;, 
 \qquad 
 \kappa(\eps) = \frac{1 - r_2\eps}{2H\Gamma(2H)}\;.
\end{equation} 
\end{proposition}
\begin{proof}
We introduce a partition $0 = t_0 < t_1 < \dots < t_N = T$ of $[0,T]$ given by 
$t_k = k\eps$ for $0 \leq k \leq N-1 = \lfloor T/\eps \rfloor$, and write 
$I_k = [t_k, t_{k+1}]$ for the $k$th interval in the partition. Then the probability on the left-hand 
side of~\eqref{eq:prop_linear_stable} is bounded by 
\begin{align}
\sum_{k=0}^{N-1} &\fP\biggl\{ \sup_{t\in I_k} |x_t| |a(t)|^H \geq h \biggr\} 
\leq \sum_{k=0}^{N-1} \fP\biggl\{ \sup_{t\in I_k} |x_t| \geq h \inf_{t\in I_k}
\frac{1}{|a(t)|^H}\biggr\} \\
&\leq 
2K_0(2H)\eps T \Bigl( \frac{h}{\sigma} \Bigr)^{1/H}
\sum_{k=0}^{N-1} \inf_{t\in I_k} \frac{1}{|a(t)|} 
\exp\biggl\{ -\frac{h^2}{2} 
\biggl( \inf_{t\in I_k} \frac{1}{|a(t)|^{2H}} \biggr) 
\biggl(\sup_{t\in I_k} \Var(x_t) \biggr)^{-1} \biggr\}\;.
\end{align}
To obtain the last line, we have applied Theorem~\ref{gaussineqforHolderCont} to $x_t$
with $\gamma=2H$ and $G$ of order $\sigma^2/\varepsilon^{2H}$, which is justified for $H>\frac12$ by~\cite[Theorem 3.7]{EKN}. For $H<\frac12$ a computation similar to the one in Lemma~\ref{variance} entails the mean-square H\"older continuity of the non-autonomous fractional Ornstein--Uhlenbeck process with the same coefficients $\gamma=2H$ and $G$ of order $\sigma^2/\varepsilon^{2H}$.
Now we observe that, setting 
\begin{equation}
 \hat v(t) = \frac{2H\Gamma(2H)}{|a(t)|^{2H}}\;,
\end{equation} 
\eqref{estimate:variance} implies $\Var(x_t) \leq \sigma^2 \hat v(t) (1 + \mathcal{O}(\eps))$
for all $t\in[0,T]$. The result thus follows from the regularity properties~\eqref{eq:assump_a} of $a$ 
and the fact that the length of $I_k$ is bounded by $\eps$.
\end{proof}

The proposition shows that as soon as $h \gg \sigma$, the sample paths $(x_t)_{t\in[0,T]}$ are unlikely 
to leave a strip of width $h|a(t)|^{-H}$ before time $T$. In other words, we obtain a \lq\lq confidence 
strip\rq\rq\ for these sample paths. 

\begin{remark}
Instead of using a partition of spacing $\eps$, one could choose a partition given by $t_k = k\delta$
for an arbitrary $\delta\in(0,T]$. This yields an extra factor $\eps/\delta$ in the prefactor $C$, 
and an additional error term of order $\delta$ in the exponent. Taking $\delta < \eps$ is not of interest, 
since it increases the prefactor while it does not improve the exponent. Otherwise, the optimal value of $\delta$ has order $\sigma^2/h^2$, and yields a prefactor of order $\eps T^2(h/\sigma)^{2+1/H}$.  Because of the condition $\delta \geq \eps$, this is only of interest if $h^2 < \sigma^2/\eps$. 
\end{remark}

\subsection{Nonlinear case}

This result can now easily be extended to concentration estimates 
for sample paths near stable slow manifolds of non-linear slowly time-dependent fractional SDEs. 
We consider the equation 
\begin{equation}
\label{eq:SDE_nonlinear} 
 \6x_t = \frac{1}{\eps} f(t,x_t) \6t 
 + \frac{\sigma}{\eps^H} \6W^H_t\;,
\end{equation} 
where $f:[0,T]\times\R\to\R$ is of class $\mathcal{C}^2$. Assume that $x^\star:[0,T]\to\R$ 
is a stable uniformly hyperbolic slow manifold (or stable equilibrium branch), meaning that 
\begin{itemize}
\item   $f(t,x^\star(t)) = 0$ for all $t\in[0,T]$,
\item   and there exists $a_0 > 0$ such that $a^\star(t) = \partial_x f(t,x^\star(t)) \leq -a_0$
for all $t\in[0,T]$. 
\end{itemize}
Note that by the implicit function theorem, such a function $x^\star$ is of class $\mathcal{C}^2$ as well. Examples of nonlinear terms are given by stable cubic nonlinearities such as $f(x)=x-x^3+a(t)$, for a smooth bounded time-dependent function, or the Stommel-Cessi model for the Atlantic Meridional Overturning Circulation~\cite{Stommel,Cessi} investigated in~\cite{EKN} and~\cite{LKN}.
Classical results by Tihonov~\cite{Tihonov} and Fenichel~\cite{Fenichel} ensure that for sufficiently small $\eps$, the deterministic equation $\eps\dot x = f(t,x)$ admits a particular solution $\bar x(t)$ satisfying $\bar x(t) = x^\star(t) + \mathcal{O}(\eps)$ uniformly in $t\in[0,T]$. 
We set 
\begin{equation}
 \bar a(t) = \partial_x f(t, \bar x(t)) 
\end{equation} 
and observe that it is bounded above by $-\bar a_0 = -a_0 + \mathcal{O}(\eps)$, which is still negative for $\eps$ small 
enough. We define the set 
\begin{equation}
 \cB(h) = \bigl\{ (t,x) \colon t\in[0,T], |x - \bar x(t)| |\bar a(t)|^H \leq h \bigr\}\;,
\end{equation} 
which is a strip of width $h|\bar a(t)|^{-H}$ around the graph of $\bar x$, and write $\tau_{\cB}(h)$ 
for the first-exit time of $(x_t)_t$ from $\cB(h)$. Then we have the following concentration result, 
which is the main result of this section. 

\begin{theorem}[Concentration estimate for the nonlinear SDE]
\label{c:sde:nl}
There exist $\eps_0, h_0 > 0$ such that for $\eps\leq\eps_0$ and $h\leq h_0$, the solution of~\eqref{eq:SDE_nonlinear} with initial condition $x_0 = \bar x(0)$ satisfies 
\begin{equation}
 \fP\bigl\{ \tau_{\cB(h)} \leq T \bigr\}
 \leq C\Bigl(T; \frac{h}{\sigma}, \bar a_0\Bigr)
 \exp\biggl\{ -\kappa(\eps)\frac{h^2}{2\sigma^2}[1 - \cO(h)]\biggr\}\;,
\end{equation} 
where the constants $C(T;\frac{h}{\sigma},\bar a_0)$ and $\kappa(\varepsilon)$ are the same as in Proposition~\ref{prop:lin:stable}.
\end{theorem}
\begin{proof}
The difference $y_t = x_t - \bar x(t)$ satisfies the SDE 
\begin{equation}
 \6y_t = \frac{1}{\eps} \bigl[ \bar a(t) + b(t,y_t) \bigr] \6t 
 + \frac{\sigma}{\eps^H} \6W^H_t\;,
\end{equation} 
for a function $b$ such that $|b(t,y)| \leq My^2$ for constants $M, d > 0$, whenever 
$t\in[0,T]$ and $|y| \leq d$. 
Its solution can be represented as $y_t = y^0_t + y^1_t$, where
\begin{equation}
 y^0_t = \frac{\sigma}{\eps^H} \int_0^t \e^{\alpha(t,s)/\eps}\6W^H_s\;, 
 \qquad
 y^1_t = \frac{1}{\eps} \int_0^t \e^{\alpha(t,s)/\eps} b(s,y_s) \6s\;.
\end{equation} 
For any decomposition $h = h^0 + h^1$ with $h^0, h^1 > 0$, continuity of sample 
paths allows us to write 
\begin{align}
\fP\bigl\{ \tau_{\cB(h)} \leq T \bigr\}
&=  \fP\biggl\{ \sup_{0\leq t\leq \tau_{\cB(h)}} |y_t| |\bar a(t)|^H \geq h \biggr\} \\
&\leq \fP\biggl\{ \sup_{0\leq t\leq \tau_{\cB(h)}} |y^0_t| |\bar a(t)|^H \geq h^0 \biggr\}
+ \fP\biggl\{ \sup_{0\leq t\leq \tau_{\cB(h)}} |y^1_t| |\bar a(t)|^H \geq h^1 \biggr\}\;.
\label{eq:decomp_SDE} 
\end{align}
Since 
\begin{equation}
 |y^1_{t\wedge\tau_\cB(h)}| \leq \frac{1}{\eps} \int_0^t \e^{\alpha(t,s)/\eps} 
 \frac{Mh^2}{|\bar a(s)|^H} \6s
 \leq \frac{Mh^2}{\bar a_0^{1+H}}\;, 
\end{equation} 
the second term on the right-hand side of~\eqref{eq:decomp_SDE} vanishes for an $h^1$ of 
order $h^2$. The result then follows from Proposition~\ref{prop:lin:stable}, taking 
$h^0 = h - h^1 = h[1 - \cO(h)]$. 
\end{proof}

\section{The SPDE case}
\label{sec:SPDE} 

\subsection{Linear case} 

We now turn to the analysis of linear SPDEs on the one-dimensional torus $\T$, of the form  
\begin{equation}
\label{spde}
 \6\phi(t,x) = \frac{1}{\eps} [\Delta + a(t)] \phi(t,x)\6t + \frac{\sigma}{\eps^H} \6W^H(t,x)\;,
\end{equation}
where $a:[0,T]\to\R$ satisfies again~\eqref{eq:assump_a}. The SPDE is driven by a cylindrical fractional Brownian motion $(W^H(t))_{t\geq 0}$ with Hurst parameter $H\in(\frac14,1)$. This means that the noise is fractional-in-time and white-in-space. The existence of mild solutions for such linear SPDEs 
was established in~\cite[Example 3.1]{DuncanMaslowski} for $H\in(\frac14,1)$. 
 
The $k$th Fourier component of $\phi$ satisfies 
\begin{equation}
 \6\phi_k(t) = -\frac{1}{\eps} \lambda_k(t) \phi_k(t)\6t + \frac{\sigma}{\eps^H} \6W^H_k(t)\;, 
\end{equation} 
where $\lambda_k(t)$ is the $k$th eigenvalue of $-\Delta - a(t)$ given by $\lambda_k(t)=(2\pi)^2 k^2 - a(t)$. It has the order $\jbrak{k}^2$, where $\jbrak{k} = \sqrt{1 + k^2}$, and satisfies 
$|\lambda_k(t)|\geq (2\pi)^2 k^2 +a_0$. 

\begin{remark}
Rescaling time as $t=\mu_k \tilde{t}$, where $\mu_k = (a_0+ck^2)^{-1}$ for a constant $c>0$, we obtain
\begin{equation}
\label{fourier:rescaled}
\6 \phi_k = -\frac{1}{\varepsilon} \tilde\lambda_k(\tilde{t})\phi_k(\tilde{t})\6 \tilde{t} + \frac{\sigma \mu_k^H}{\varepsilon^H}\6 W^H_k (\tilde{t})
\end{equation}
with $\tilde t\in[0, T/\mu_k]$, 
where $\tilde\lambda_k(\tilde{t}) = \mu_k\lambda_k(\tilde{t}/\mu_k)$, and consequently 
$\tilde\lambda_k(\tilde{t})\geq 1$ and $|\tilde\lambda_k'(\tilde{t})| \leq a_1$ for all $\tilde{t}\in[0,T/\mu_k]$. This allows us to use Proposition~\ref{prop:lin:stable} with values of $a_0$ and $a_1$ that do not depend on $k$. Therefore, we need not worry about a possible $k$-dependence of the error term 
$r_2\eps$ in~\eqref{eq:factors_linear_stable}.
\end{remark}

We recall that the (fractional) Sobolev norm on $H^s(\T)$ for $s>0$ is given by
\begin{equation}\label{fr:sobolev}
\|\phi(t,\cdot)\|_{H^s}^2 =\sum\limits_{k\in\Z} \jbrak{k}^{2s}|\phi_k(t)|^2\;.
\end{equation}
While one can work with this norm, it turns out that we can obtain slightly sharper 
bounds using a time-dependent Sobolev norm defined by
\begin{equation}
\|\phi(t,\cdot)\|_{s,t}^2 := \sum\limits_{k\in\Z} a_{k,s}(t)^2|\phi_k|^2\;,  
\end{equation}
where we will choose $a_{k,s}(t)=\lambda^H_k(t)\jbrak{k}^{s-2H}$, so that 
$a_{k,s}(t)\asymp \jbrak{k}^s$. Note that both norms are equivalent, and are not sensitive to 
Fourier modes with large $|k|$. However, the time-dependent norm will give a sharper control 
for Fourier modes with small $|k|$, and especially for $k = 0$.  

\begin{proposition}[Concentration estimate for the linear SPDE]
\label{c:spde}
Let $H\in(\frac14,1)$ and $s\in(0,2H-\frac12)$. Then there exist constants $c_0, c_1, r_2>0$ such that 
the solution of~\eqref{spde} with initial condition $\phi(0,\cdot) = 0$ satisfies the concentration inequality
\begin{equation}
\fP\biggl\{ \sup_{0 \leq t \leq T} \|\phi(t,\cdot)\|_{s,t}  \geq h \biggr\} 
\leq C\Bigl(T;\frac{h}{\sigma},s\Bigr) \exp\biggl\{ -\kappa(\varepsilon) Q(s)\frac{h^2}{2\sigma^2 }\biggr\}\;, 
\end{equation}
where $Q(s)$, defined in~\eqref{Q} below, satisfies 
$Q(s) \geq c_0(2H - \frac12 - s)$,
$\kappa(\eps)$ is the same as in~\eqref{eq:factors_linear_stable}, and 
\begin{equation}
 C\Bigl(T;\frac{h}{\sigma},s\Bigr)
 = 2K_0(2H)T^2 a_0^2 
 \Bigl(Q(s)^{1/2}\frac{h}{\sigma}\Bigr)^{1/H} 
 \bigl[1 + \cO(\e^{-c_1h^2/\sigma^2})\bigr]\;.
\end{equation} 
\end{proposition}
\begin{proof}
It is known that the fractional stochastic convolution has continuous trajectories in $H^s$ 
for $s\in(0,2H-\frac12)$, see~\cite[Corollary 3.1]{DuncanMaslowski} for $H>\frac12$, 
respectively~\cite[Lemma 11.10]{DuncanMaslowski2} for $H\in(\frac14,\frac12)$. 
For any decomposition $h^2=\sum_{k\in\Z} h^2_k$ with $h_k > 0$ for all $k\in\Z$, we have 
\begin{align}
\fP\biggl\{ \sup_{0 \leq t \leq T} \|\phi(t,\cdot)\|_{s,t} \geq h \biggr\}
&= \fP\biggl\{ \sup_{0 \leq t \leq T} \sum\limits_{k\in\Z} a_{k,s}(t)^2|\phi_k(t)|^2 \geq h^2 \biggr\}\\
&\leq \sum\limits_{k\in\Z} \fP\biggl\{ \sup_{0 \leq t \leq T} |\phi_k(t)|^2\lambda_k(t)^{2H} \geq \frac{h_k^2}{\jbrak{k}^{2s-4H}} \biggr\}
\label{eq:proof_SPDE_lin1} 
\end{align}
by the choice of the time-dependent coefficients $a_{k,s}(t) =\lambda^H_k(t)\jbrak{k}^{s-2H}$. 
According to Proposition~\ref{prop:lin:stable}, we have for each component $\phi_k$ solving~\eqref{fourier:rescaled} that
\begin{align}
\fP\biggl\{ \sup_{0 \leq t \leq T} |\phi_k(t)| \lambda_k(t)^{H} 
\geq \frac{h_k}{\jbrak{k}^{s-2H}} \biggr\} 
&= \fP\biggl\{ \sup_{0 \leq \tilde{t} \leq T/\mu_k} 
|\phi_k(\tilde{t})| \frac{\tilde\lambda_k(\tilde{t})^{H}}{\mu_k^H}
\geq  \frac{h_k}{\jbrak{k}^{s-2H}} \biggr\}\\
& \leq C_k 
\exp \biggl\{ -\kappa(\varepsilon) 
\frac{h^2_k\mu_k^{2H}}{2\tilde\sigma_k^2\jbrak{k}^{2s-4H}} \biggr\}\;,
\end{align}
where $\tilde\sigma_k = \mu_k^H\sigma$ due to the scaling in~\eqref{fourier:rescaled}, and 
\begin{equation} 
C_k = C\biggl(\frac{T}{\mu_k},\frac{h_k\mu_k^H}{\tilde\sigma_k\jbrak{k}^{s-2H}},1\biggr) 
= \frac{2K_0 (2H) T^2}{\mu_k^2}
\bigg(\frac{h_k}{\sigma\jbrak{k}^{s-2H}}\bigg)^{1/H}\;.
\end{equation}
Plugging this in~\eqref{eq:proof_SPDE_lin1} and simplifying the factors $\mu_K^H$ entails 
\begin{equation}
 \fP\biggl\{ \sup_{0 \leq t \leq T} \|\phi(t,\cdot)\|_{s,t}  \geq h \biggr\} 
 \leq  \sum\limits_{k\in\Z} 
 C_k\exp\biggl\{ -\kappa(\varepsilon) \frac{h_k^2}{2\sigma^2 \jbrak{k}^{2s-4H} }\biggr\}\;.    
\end{equation}
We pick $\eta>0$ and choose 
$h^2_k=Q(s)h^2\jbrak{k}^{-(4H-2s-\eta)}$, where the condition 
$h^2=\sum_{k\in\Z} h^2_k$ imposes
\begin{equation}\label{Q} 
Q(s)^{-1} =
\sum\limits_{k\in\Z} \frac{1}{\jbrak{k}^{4H-2s-\eta}} < \infty\;. 
\end{equation}
This means that we need to have $4H-2s-\eta>1$, and since both $s$ and $\eta$ must be positive we 
obtain the restriction $H>\frac14$. The claimed lower bound on $Q(s)$ follows from the behaviour of 
Riemann's zeta function $\zeta(u)$ as $u\to 1$, choosing for instance $\eta = 2H - s -\frac12$. 
Based on this choice of the $h_k$, we further obtain 
\begin{align}
\fP\biggl\{ \sup_{0 \leq t \leq T} \|\phi(t,\cdot)\|_{s,t}  \geq h \biggr\} 
&\leq \sum\limits_{k\in\Z} C_k
\exp\biggl\{ -\kappa(\varepsilon)Q(s) 
\frac{h^2\jbrak{k}^{\eta}}{2\sigma^2  }\biggr\}\\
& \leq 2K_0(2H) T^2 \biggl(Q(s)^{1/2}\frac{h}{\sigma}\bigg)^{1/H}
\sum\limits_{k\in\Z} 
\frac{\jbrak{k}^{\eta/(2H)}}{\mu_k^2}
\exp\bigl\{-\beta \jbrak{k}^\eta\bigr\}\;,
\end{align}
where $\beta:=\kappa(\varepsilon) Q(s)\frac{h^2}{2\sigma^2}$. 
We claim that the sum over $k$ is dominated by the term $k=0$. 
In fact, by an argument similar to the one in~\cite[Theorem 2.4]{BN23}, we can bound this sum by 
\begin{equation}
 f(0)+ 2f(1) + \int_1^\infty f(x)\6 x\;, 
 \qquad\text{where}\qquad 
 f(x) := (a_0+cx^2)(1+x^2)^\gamma\e^{-\beta(1+x^2)^{\eta/2}}\;,
\end{equation} 
with an exponent 
$\gamma = \eta/(4H)$. The integral can be shown to be of order $\e^{-c\beta}$ 
with $c>1$, which yields the result. 
\end{proof}

\begin{remark}
The condition $H>\frac14$ is consistent with the solution theory for such SPDEs driven by cylindrical fractional noise which is required in order to define the stochastic convolution~\cite{DuncanMaslowski,TTV}. 
\end{remark}

\subsection{Nonlinear case}

We finally consider the nonlinear SPDE given by 
\begin{equation}
\label{spde:nonlinear}
 \6\phi(t,x) = \frac{1}{\eps} [\Delta \phi(t,x) + f(t,\phi(t,x))]\6t+ \frac{\sigma}{\eps^H} \6W^H(t,x)\;.
\end{equation}
We assume $f(t,\phi)=-\partial_\phi U(t,\phi)$, where the potential $U$ can be decomposed as $U(t,\phi)=P(t,\phi)+g(t,\phi)$. Here $P$ is polynomial of even degree $2p$ with smooth bounded coefficients such that the leading order coefficient $a_{2p}(t)>0$ for all $t\in[0,T]$, and the function $g\in \cC^2([0,T]\times \R;\R)$ satisfies the boundedness assumptions 
\begin{equation}
|g(t,\phi)\phi^{-1}|,\;
|\partial_\phi g(t,\phi)|,\;
|\partial_{\phi\phi}g(t,\phi)|,\; 
|\partial_t g(t,\phi)|\leq \widetilde{M}
\end{equation} 
for all $(t,\phi)\in[0,T]\times \R$, for a constant $\widetilde{M}>0$. 

\begin{remark}
The well-posedness of semilinear SPDEs with dissipative drift terms was established for $H<1/2$ in~\cite[Theorem 3.6]{DMD3} and for $H>1/2$ in~\cite[Theorem 3.7]{DMD3}. Provided that the trace of the covariance operator $\mathcal{K}$ of the fractional Brownian motion is finite, a mild solution exists for all ranges of $H\in(0,1)$. In our case, we assume that the noise is fractional in time and white in space, i.e. $\mathcal{K}=\text{Id}$, and need the restriction $H>1/4$. This is required in order to define the stochastic convolution, i.e.~the mild solution of the linear equation. The proof in the semilinear case relies on a standard fixed-point argument in $C([0,T];L^2(\T))$, subtracting the stochastic convolution and using that this belongs to $C([0,T];D((-\Delta)^q))$ for $0\leq q<H$ according to~\cite[Proposition 2.6]{DMD3}. 
A sketch of this argument was provided for stable cubic nonlinearities in~\cite[Section 3]{BlBl23}. \\
The well-posedness result is applicable  to drift terms of the form $U(t,\phi)=P(t,\phi)+g(t,\phi)$, due to the smoothness and boundedness assumptions on $g$. 
\end{remark}

As before, we assume that we are in a stable situation, meaning that there exists a map $\phi^*:[0,T]\to \R$ such that 

\begin{itemize}
\item $f(t,\phi^*(t))=0$ for all $t\in[0,T]$,
\item there exists $a_0>0$ such that  $a^*(t)=\partial_\phi f(t,\phi^*(t))\leq -a_0$ for all $t\in[0,T]$. 
\end{itemize}

Similarly to the finite-dimensional case, the deterministic PDE admits, according to~\cite[Proposition 2.3]{BN23}, a solution $\bar\phi$ such that $\bar\phi(t)=\phi^*(t)+\cO(\varepsilon)$ for all $t\in[0,T]$. We set as before
$\bar a(t) = \partial_x f(t, \bar \phi(t))$  
and introduce for $s\in(0,2H-\frac12)$
\begin{equation}
\cB(h):= \bigl\{ (t,\phi): t\in[0,T], \|\phi(t,\cdot)-\bar\phi (t,\cdot)\|_{s,t}\leq h \bigr\}\:. 
\end{equation}
The SPDE for the difference $\psi(t,x):=\phi(t,x)-\bar{\phi}(t,x)$ reads
\begin{equation}\label{SPDE:NL}
\6\psi(t,x) = \frac{1}{\eps} [\Delta \psi(t,x) + \bar{a}(t) \psi(t,x) + b(t,\psi(t,x))]\6t+ \frac{\sigma}{\eps^H} \6W^H(t,x)\;, 
\end{equation}
for a function $b$ satisfying, for constants $M,d>0$, $|b(t,\psi)|\leq M \psi^2$ and $|\partial_\psi b(t,\psi)|\leq M|\psi|$ for all $t\in[0,T]$ and $\psi\in\R$ with $|\psi|\leq d$.
By the variation of constants formula, its mild solution is given by 
\begin{align}
\psi(t,x) &= \frac{\sigma}{\varepsilon^H}\int_0^t \e^{\alpha(t,s)/\varepsilon} \e^{[(t-s)/\varepsilon]\Delta}~\6 W^{H}(s) +   \frac{1}{\varepsilon}\int_0^t \e^{\alpha(t,s)/\varepsilon}\e^{[(t-s)/\varepsilon]\Delta}b(s,\psi(s,x))~\6s\\
& := \psi^0(t,x) + \psi^1(t,x)\;.
\end{align}
As before, $\alpha(t,s)=\int_s^t\bar{a}(r)\6 r$, while $\e^{\cdot \Delta}$ is the heat semigroup, 
given by spatial convolution with the heat kernel. 
In order to analyze the stochastic convolution, we rely on Schauder-type estimates. 

\begin{lemma}[Schauder-type estimates]
Let $f\in H^r$ with $r\in(0,2H-\frac12)$. Then for all $q<r+2$, there exists a constant $c(q,r)>0$ such that
\begin{equation}\label{schauder}
  \| \e^{t\Delta} f\|_{H^q} \leq c(q,r) t^{-\frac{q-r}{2}} \|f\|_{H^r}\;.
\end{equation}
\end{lemma}
\begin{proof}
This follows from regularizing properties of analytic semigroups~\cite[Theorem 6.13]{Pazy},
according to which
\begin{equation} 
\|\e^{t\Delta} f\|_{D((-\Delta)^{q/2})} \leq c(q,r) t^{-\frac{q-r}{2}}\|f\|_{D((-\Delta)^{r/2})}
\end{equation}
for $q-r/2<1$, leading to the restriction $q<r+2$. The statement follows, considering that $H^q=D((-\Delta)^{q/2})$. Here we use the equivalence of the time-dependent (fractional) Sobolev norm with~\eqref{fr:sobolev}. 
\end{proof}
For $\psi(t,\cdot)\in H^s$, one can easily prove that $\beta(t):=b(t,\psi(t,\cdot))$ belongs to $H^s$. A proof of this statement in $H^s$ for $s\in(0,\frac12)$ relying on Young's inequality is provided in~\cite[Lemma 3.4]{BN23}. We now apply the Schauder estimate in order to obtain a bound on $\psi^1(t,\cdot)$ similar to~\cite[Corollary 3.6]{BN23}.

\begin{lemma}
\label{lem:psi1} 
Assume that there exists $r\in(0,2H-\frac12)$ such that $\beta(t)\in H^r$ for all $t\in[0,T]$. Then for all $q<r+2$ there exists a constant $c'(q,r)>0$ such that for all $t\in[0,T]$ we have
\begin{equation}  
\| \psi^{1}(t,\cdot) \|_{H^q} \leq c'(q,r) \varepsilon^{\frac{q-r}{2}-1}\sup\limits_{0\leq s\leq t} \|\beta(s)\|_{H^r}\;.
\end{equation}
\end{lemma}
\begin{proof}
We have
\begin{align}
\| \psi^1(t,x)\|_{H^q} &\leq \frac{1}{\varepsilon} 
\int_0^t \e^{-a_0(t-s)/\varepsilon} \|\e^{[(t-s)/\varepsilon]\Delta}\beta(s)\|_{H^q}~\6s\\
& \leq c(q,r) \varepsilon^{\frac{q-r}{2}-1} 
\sup\limits_{0\leq s \leq t } \|\beta(s)\|_{H^r} \int_0^t (t-s)^{-\frac{q-r}{2}}~\6s<\infty\;, 
\end{align}
since $q<r+2$. In the last step we used the uniform negative bound on $\bar a$ and the Schauder estimate. The result follows with $c' = c \int_0^t (t-s)^{-\frac{q-r}{2}}~\6s$.
\end{proof}

\begin{theorem}[Concentration estimate for the nonlinear SPDE]
\label{thm:nonlinear_concentration} 
For every $s\in(0,2H-\frac12)$ and any $\nu>0$, there exist positive constants $\varepsilon_0, h_0$ 
such that for $\eps\leq\eps_0$ and $h\leq h_0 \varepsilon^\nu$, the solution of~\eqref{SPDE:NL} with initial condition $\phi(0,\cdot) = \bar \phi(0,\cdot)$ satisfies 
\begin{equation}
 \fP\bigl\{ \tau_{\cB(h)} \leq T \bigr\}
 \leq C\Bigl(T; \frac{h}{\sigma}, s\Bigr)
 \exp\biggl\{ -\kappa(\eps)Q(s)\frac{h^2}{2\sigma^2}
 \bigg[1 - \cO\Big(\frac{h}{\varepsilon^\nu}\Big)\bigg]\biggr\}\;, 
\end{equation}
with the same $C\bigl(T; \frac{h}{\sigma}, s\bigr)$, $Q(s)$ and $\kappa(\eps)$ as in 
Proposition~\ref{c:spde}. 
\end{theorem}
\begin{proof}
 For any decomposition $h=h_0+h_1$ we have
\begin{align}
 \fP\bigl\{ \tau_{\cB(h)} \leq T  \bigr\}
&=  \fP\biggl\{ \sup_{0\leq t \leq \tau_{\cB(h)}} \|\psi(t,\cdot)\|_{s,t} \geq h \biggr\}\nonumber\\
& \leq \fP\biggl\{ \sup_{0\leq t\leq \tau_{\cB(h)}} \Bigl[\|\psi^0(t,\cdot)\|_{s,t} + \|\psi^1(t,\cdot)\|_{s,t}\Bigr] \geq h \biggr\}\nonumber\\
& \leq \fP\biggl\{ \sup_{0\leq t\leq \tau_{\cB(h)} } \|\psi^0(t,\cdot)\|_{s,t} \geq h_0 \biggr\} + \fP\biggl\{ \sup_{0\leq t\leq \tau_{\cB(h)}} \|\psi^1(t,\cdot)\|_{s,t} \geq h_1 \biggr\}\;.\label{pb:nl}
\end{align}
We bound the first term using Proposition~\ref{c:spde}. For the second one we have, 
using Lemma \ref{lem:psi1} and otherwise
similarly to the proof of Theorem~\ref{c:sde:nl}, that for $t\leq \tau_{\cB(h)}$ 
\begin{equation} 
\|\psi^1(t,\cdot)\|_{H^q} \leq c'(q,r) \eps^{\frac{q-r}{2}-1}Mh^2\;, 
\end{equation} 
since $\|\beta(t)\|_{H^r}\leq M\|\psi^1(t,\cdot)\|^2_{H^q}\leq M h^2$. Therefore the second term in~\eqref{pb:nl} vanishes provided that $h_1=c'(q,r)\varepsilon^{\frac{q-r}{2}-1}Mh^2$. The statement follows, choosing 
\begin{equation}
 h_0=h-h_1=h-c'(q,r) \varepsilon^{\frac{q-r}{2}-1}Mh^2=h(1-\cO(h/\varepsilon^\nu))
\end{equation} 
for $\nu=1-\frac{q-r}{2}$.
\end{proof}


\section*{Acknowledgements} 
We thank the referees for their valuable comments and suggestions. 
AB acknowledges support by the DFG grant 543163250, DFG CRC-TRR 388 project A06 and DFG CRC 1432 project C10.
NB was supported by the ANR project PERISTOCH, ANR--19--CE40--0023. Furthermore, 
NB thanks the Department of Mathematics and Statistics at University of Konstanz, 
and AB thanks the Institut Denis Poisson at University of Orl\'eans for hospitality 
during mutual visits. 





\bigskip\bigskip\noindent
\small
\textbf{Nils Berglund} \\
Institut Denis Poisson (IDP) \\ 
Universit\'e d'Orl\'eans, Universit\'e de Tours, CNRS -- UMR 7013 \\
B\^atiment de Math\'ematiques, B.P. 6759\\
45067~Orl\'eans Cedex 2, France \\
{\it E-mail address: }
{\tt nils.berglund@univ-orleans.fr} 
%
\hspace{2em}

\bigskip\noindent 
\small
\textbf{Alexandra Blessing~(Neam\c tu)} \\
Department of Mathematics and Statistics \\ 
University of Konstanz \\
Germany \\
{\it E-mail address: }
{\tt alexandra.blessing@uni-konstanz.de}

\end{document}